\documentclass[11pt,twoside]{amsart}

\numberwithin{equation}{section}
\usepackage[table]{xcolor}
\usepackage{latexsym,amssymb,amsmath,amsthm,amsfonts}
\usepackage[cmtip,arrow]{xy}
\usepackage{pb-diagram,pb-xy}
\usepackage{graphicx}
\usepackage{tikz}
\usepackage{multicol}
\usepackage{multirow}
\usepackage{enumerate}
\usetikzlibrary{matrix,arrows}
\definecolor{VerdeOlivo}{rgb}{0.3,0.5,0.1}
\definecolor{Magenta}{rgb}{.65,0.15,.2}
\definecolor{Gris}{gray}{0.7}

\newtheorem{Theorem}{Theorem}[section] 
\newtheorem{Definition}[Theorem]{Definition} 
\newtheorem{Lemma}[Theorem]{Lemma} 
\newtheorem{Corollary}[Theorem]{Corollary} 
\newtheorem{Conjecture}[Theorem]{Conjecture} 
 
\newtheorem{Example}[Theorem]{Example} 
 
\newtheorem{Proposition}[Theorem]{Proposition} 
\newtheorem{Claim}[Theorem]{Claim}

\theoremstyle{definition}

\newcommand{\Gaa}{\sf cricket} 
\newcommand{\Gab}{\sf dart} 

\begin{document} 


\title[Graphs with two trivial critical ideals]{Graphs with two trivial critical ideals}


\author{Carlos A.  Alfaro}
\thanks{Both authors were partially supported by CONACyT grant 166059, the first author was partially supported by CONACyT and the second author was partially supported by SNI}

\author{Carlos E. Valencia}
\address{
Departamento de
Matem\'aticas\\
Centro de Investigaci\'on y de Estudios Avanzados del
IPN\\
Apartado Postal
14--740 \\
07000 Mexico City, D.F. 
} 
\email{alfaromontufar@gmail.com and cvalencia@math.cinvestav.edu.mx, respectively.}



\maketitle

\begin{abstract}
The critical ideals of a graph are the determinantal 
ideals of the generalized Laplacian matrix associated to a graph.
A basic property of the critical ideals of graphs asserts that 
the graphs with at most $k$ trivial critical ideals, 
$\Gamma_{\leq k}$, are closed under induced subgraphs. 
In this article we find the set of minimal forbidden subgraphs for $\Gamma_{\leq 2}$, 
and we use this forbidden subgraphs to get a classification of the graphs in $\Gamma_{\leq 2}$.
As a consequence we give a classification of the simple graphs whose critical group has two invariant factors equal to one.
At the end of this article we give two infinite families of forbidden subgraphs.
\end{abstract}

\section{Introduction}
Given a connected graph $G=(V(G),E(G))$ and a set of indeterminates $X_G=\{x_u \, | \, u\in V(G)\}$, 
the {\it generalized Laplacian matrix} $L(G,X_G)$ of $G$
is the matrix with rows and columns indexed by the vertices of $G$ given by
\[
L(G,X_G)_{u v}=
\begin{cases}
	x_u & \text{ if } u=v,\\
	-m_{u v} & \text{ otherwise},
	\end{cases}
\]
where $m_{u v}$ is the {\it multiplicity} of the edge $uv$, that is, the number of the edges between vertices $u$ and $v$ of $G$.
For all $1\leq i \leq n$, the $i$-{\it critical ideal} of $G$ is the determinantal ideal given by
\[
I_i(G,X_G)=\langle  \{ {\rm det} (m) \, | \, m \text{ is a square submatrix of }L(G,X_G) \text{ of size } i\}\rangle\subseteq \mathbb{Z}[X_G].
\]
We say that a critical ideal is trivial when it is equal to $\langle 1 \rangle$.

Critical ideals are a generalization of the characteristic polynomials of the adjacency matrix and the Laplacian matrix associated to a graph.
Also, critical ideals generalize the critical group of a graph as shown below:
if $d_{G}(u)$ is the degree of a vertex $u$ of $G$, then 
the {\it Laplacian matrix} of $G$, denoted by $L(G)$, is the evaluation of $L(G,X_G)$ on $x_u=d_G(u)$.
Given a vertex $s$ of $G$, the {\it reduced Laplacian matrix} of $G$, denoted by $L(G,s)$, is the matrix obtained from $L(G)$ by removing the row and column $s$.
The {\it critical group} of a connected graph $G$, denoted by $K(G)$, is the cokernel of $L(G,s)$.
That is, 
\[
K(G)=\mathbb{Z}^{\widetilde{V}}/{\rm Im}\, L(G,s),
\]
where $\widetilde{V}=V(G)\setminus s$.
Therefore the critical group of a graph can be obtained from their critical ideals as shows~\cite[theorems 3.6 and 3.7]{corrval}.
The critical group have been studied intensively on several contexts over the last 30 years. 
However, a well understanding of the combinatorial and algebraic nature of the critical group still remains.
				
Let assume that $G$ is a connected graph with $n$ vertices.
A classical result (see \cite[section 3.7]{jacobson}) asserts that the reduced Laplacian matrix is equivalent to a 
integer diagonal matrix with entries $d_1, d_2, ..., d_{n-1}$ where $d_i> 0$ and $d_i \mid d_j$ if $i\leq j$.
The integers $d_1, \dots, d_{n-1}$ are unique and are called {\it invariant factors}. 
With this in mind, the critical group is described in terms of the invariant factors as follows \cite[theorem 1.4]{lorenzini1989}:
\[
K(G)\cong \mathbb{Z}_{d_1} \oplus \mathbb{Z}_{d_2} \oplus \cdots \oplus\mathbb{Z}_{d_{n-1}}.
\]
Given an integer number $k$, let $f_k(G)$ be the number of invariant factors of the Laplacian matrix of $G$ equal to $k$.
Let $\mathcal{G}_i=\{G\, |\, G \text{ is a simple connected graph with } f_1(G)=i\}$.
The study and characterization of $\mathcal G_i$ is of great interest.
In particular, some results and conjectures on the graphs with cyclic critical group can be found in \cite[section 4]{lorenzini2008} and \cite[conjectures 4.3 and 4.4]{wagner}.
On the other hand, Dino Lorenzini, notice in \cite{lorenzini1991} that $\mathcal G_1$ consists only of the complete graphs.
More recently, Merino in \cite{merino} posed interest on the characterization of $\mathcal G_2$ and $\mathcal G_3$. 
In this sense, few attempts have done.
For instance, in \cite{pan} it was characterized the graphs in $\mathcal G_2$ whose third invariant factor is equal to $n$, $n-1$, $n-2$, or $n-3$.
In \cite{chan} the characterizations of the graphs in $\mathcal G_2$ with a cut vertex, 
and the graphs in $\mathcal G_2$ with number of independent cycles equal to $n-2$ are given.

\medskip

If $\Gamma_{\leq i}$ denotes the family of graphs with at most $i$ trivial critical ideals,
then it is not difficult to see that $\mathcal{G}_i\subseteq \Gamma_{\leq i}$ for all $i\geq 0$.
At first glance, critical ideals behave better than critical ideals.
For instance, by~\cite[proposition 3.3]{corrval} we have that $\Gamma_{\leq i}$ is closed under induced subgraphs at difference of $\mathcal{G}_i$.
This property will play a crucial role in order to get a characterization of $\Gamma_{\leq 2}$ on this paper.
Also, if $\Gamma_{i}$ is the family of graphs with exactly $i$ trivial critical ideals,
then we will shown on this paper that $\Gamma_{2}$ has a more simple description that $\mathcal{G}_2$.

\medskip

The main goals of this paper are three: to get a characterization of the graphs with at most two trivial critical ideals,
to get a characterization of the graphs with two invariant factors equal to one, and to give two infinite families of forbidden subgraphs for $\Gamma_{\leq i}$.

This article is divided as follows:
We begin by recalling some basic concepts on graph theory in section 2 and establishing some of basic properties of critical ideals in section 3.
In section 4 we will characterize the graphs with at most two trivial critical ideals by finding their minimal set of forbidden graphs. 
As consequence, we will get the characterization of the graphs with two invariant factors equal to one.
Finally, in section 5 we give two infinite families of forbidden graphs for $\Gamma_{\leq i}$.

\section{Basic definitions}
In this section, we give some basic definitions and notation of graph theory used in later sections.
It should be pointed that we will consider the natural number as the the non-negative integers.

Given a graph $G=(V,E)$ and a subset $U$ of $V$, the subgraph of $G$ {\it induced} by $U$ will be denoted by $G[U]$.
If $u$ is a vertex of $G$, let $N_G(u)$ be the set of {\it neighbors} of $u$ in $G$.
Here a {\it clique} of a graph $G$ is a maximum complete subgraph, and its order is the {\it clique number} of $G$, denoted by $\omega(G)$.
The {\it path} with $n$ vertices is denoted by $P_n$, a {\it matching} with $k$ edges by $M_k$, 
the {\it complete graph} with $n$ vertices by $K_n$ and the {\it trivial graph} of $n$ vertices by $T_n$.
The {\it cone} of a graph $G$ is the graph obtained from $G$ by adding a new vertex, called {\it appex}, which is adjacent to each vertex of $G$.
The cone of a graph $G$ is denoted by $c(G)$.
Thus, the {\it star} $S_k$ of $k+1$ vertices is equal to $c(T_k)$.
Given two graphs $G$ and $H$, their {\it union} is denoted by $G\cup H$, and their {\it disjoint union} by $G+H$.
The {\it join} of $G$ and $H$, denoted by $G\vee H$, is the graph obtained from $G+H$ when we add all the edges between vertices of $G$ and $H$.
For $m,n,o\geq 1$, let $K_{m,n,o}$ be the {\it complete tripartite graph}. 
You can consult~\cite{diestel} for any unexplained concept of graph theory.

Let $M\in M_{n}(\mathbb{Z})$ be a $n\times n$ matrix with entries on 
$\mathbb{Z}$, $\mathcal{I}=\{i_1,\ldots,i_r\}\subseteq \{1,\ldots,n\}$, and $\mathcal{J}=\{j_1,\ldots,j_s\}\subseteq \{1,\ldots,n\}$.
The submatrix of $M$ formed by rows $i_1,\ldots,i_r$ and columns $j_1,\ldots,j_s$ is denoted by $M[\mathcal{I};\mathcal{J}]$.
If $|\mathcal{I}|=|\mathcal{J}|=r$, then $M[\mathcal{I};\mathcal{J}]$ is called a {\it $r$-square submatrix} or a {\it square submatrix} of size $r$ of $M$.
A {\it $r$-minor} is the determinant of a $r$-square submatrix.
The set of $i$-minors of a matrix $M$ will be denoted by ${\rm minors}_i(M)$.   
Finally, the identity matrix of size $n$ is denoted by $I_n$ and the all ones $m\times n$ matrix is denoted by $J_{m,n}$.
We say that $M,N \in M_{n}(\mathbb{Z})$ 
are {\it equivalent}, denoted by $N\sim M$, if there exist $P,Q\in GL_n(\mathbb{Z})$ such that $N=PMQ$.
Note that if $N\sim M$, then $K(M)=\mathbb{Z}^n/M^t\mathbb{Z}^n\cong\mathbb{Z}^n/N^t\mathbb{Z}^n=K(N)$.

\section{Graphs with few trivial critical ideals}
In this section, we will introduce the critical ideals of a graph and the set of 
graphs with $k$ or less trivial critical ideals, denoted by $\Gamma_{\leq k}$.
After that, we define the set of minimal forbidden graphs of $\Gamma_{\leq k}$.
We finish this section with the classification of $\mathcal G_1$, that we already know that they are the complete graphs.

Let $G$ be a graph and $X_G=\{x_v \mid v\in V(G)\}$ be the set of indeterminates indexed by the vertices of $G$.
For all $1\leq i \leq n$, the {\it $i$-critical ideal} $I_i(G, X_G)$ is defined as the ideal of $\mathbb{Z}[X_G]$ given by
\[
I_i(G,X_G)=\langle  \{ {\rm det} (m) \, | \, m \text{ is a square matrix of }L(G,X_G) \text{ of size } i\}\rangle.
\]
By convention $I_i(G,X_G)=\left< 1\right>$ if $i<1$, and $I_i(G,X_G)=\left< 0\right>$ if $i>n$.
The {\it algebraic co-rank} of $G$, denoted by $\gamma(G)$, is the number of critical ideals of $G$ equal to $\left<1\right>$.

\begin{Definition}
For all $k\in \mathbb{N}$, let $\Gamma_{\leq k}=\{ G\, | \, G \text{ is a simple connected graph with } \gamma(G)\leq k \}$ 
and $\Gamma_{\geq k}=\{ G\, | \, G \text{ is a simple connected graph with }  \gamma(G)\geq k \}$.
\end{Definition}

Note that, $\Gamma_{\leq k}$ and $\Gamma_{\geq k+1}$ are disjoint sets and that a characterization 
of one of them leads to a characterization of the other one.
Now, let us recall some basic properties about critical ideals, see \cite{corrval} for details.
It is known that if $i\leq j$, then $I_j(G,X_G) \subseteq I_i(G,X_G)$.
Moreover, if $H$ is an induced subgraph of $G$, then $I_i(H,X_H)\subseteq I_i(G,X_G)$ 
for all $i\leq |V(H)|$ and therefore $\gamma(H)\leq \gamma(G)$.
This implies that $\Gamma_{\leq k}$ is closed under induced subgraphs, that is, 
if $G\in \Gamma_{\leq k}$ and $H$ is an induced subgraph of $G$, then $H\in \Gamma_{\leq k}$.

\begin{Definition}
Let $f_k(G)$ be the number of invariant factors of $K(G)$ that are equal to $k$ and
\[
\mathcal{G}_i=\{G\, |\, G \text{ is a simple connected graph with } f_1(G)=i\}.
\]
\end{Definition}

Presumably $\Gamma_{\leq k}$ behaves better than $\mathcal G_k$.
It is not difficult to see that unlike of $\Gamma_{\leq k}$, $\mathcal G_k$ is not closed under induced subgraphs.
For instance, considerer $c(S_3)$, clearly it belongs to $\mathcal G_2$, but $S_3$ belongs to $\mathcal G_3$. 
Similarly, $K_6\setminus \{ 2P_2\}$ belongs to $\mathcal G_3$ meanwhile $K_5\setminus \{ 2P_2\}$ belongs to $\mathcal G_2$.
Moreover, if $H$ is an induced subgraph of $G$, it is not always true that $K(H)\trianglelefteq K(G)$.
For example, $K(K_4)\cong \mathbb{Z}_4^2\ntrianglelefteq K(K_5)\cong \mathbb{Z}_5^3$.
Finally, theorems \ref{teo:main1} and \ref{teo:main2} gives us additional evidence 
in the sense that  $\Gamma_{\leq k}$ behaves better than $\mathcal G_k$.
Moreover, theorem 3.6 of ~\cite{corrval} implies that $\gamma(G)\leq f_1(G)$ for any graph
and therefore $\mathcal{G}_k \subseteq \Gamma_{\leq k}$ for all $k\geq 0$.

\medskip

A graph $G$ is \textit{forbidden} (or an \textit{obstruction}) for $\Gamma_{\leq k}$ if and only if $\gamma(G)\geq k+1$.
Let ${\bf Forb}(\Gamma_{\leq k})$ be the set of minimal (under induced subgraphs property) forbidden graphs for $\Gamma_{\leq k}$.
Also, a graph $G$ is called $\gamma$-\textit{critical} if $\gamma(G\setminus v)< \gamma(G)$ for all $v\in V(G)$.
That is, $G\in {\bf Forb}(\Gamma_{\leq k})$ if and only if $G$ is $\gamma$-critical with $\gamma(G)=k+1$.

Given a family of graphs $\mathcal F$, a graph $G$ is called $\mathcal F$-free if no induced subgraph of $G$ is isomorphic to a member of $\mathcal F$.
Thus, 
$G$ belongs to $\Gamma_{\leq k}$ if and only if $G$ is ${\bf Forb}(\Gamma_{\leq k})$-free, or equivalently,
$G$ belongs to $\Gamma_{\geq k+1}$ if and only if $G$ contains a graph of ${\bf Forb}(\Gamma_{\leq k})$ as an induced subgraph.

These ideas are useful to characterize $\Gamma_{\leq k}$.
For instance, since $\gamma(P_2)=1$ and no one of its induced subgraphs has $\gamma \geq 1$, then $P_2\in {\bf Forb}(\Gamma_{\leq 0})$.
Moreover, it is easy to see that $T_1$ is the only connected graph that is $P_2$-free.
Thus, since $I_1(T_1,\{ x\})\neq \left< 1\right>$, we get that ${\bf Forb}(\Gamma_{\leq 0})=\{ P_2\}$, 
and $\Gamma_{\leq 0}$ consists of the graph with one vertex.
Also, it is not difficult to prove that $\mathcal G_0=\Gamma_{\leq 0}$ and
that the set of non-necessarily connected graphs with algebraic co-rank equal to zero consists only of the trivial graphs.
In the next section we will use this kind of arguments in order to get ${\bf Forb}(\Gamma_{\leq k})$
and characterize $\Gamma_{\leq k}$ for $k$ equal to $1$ and $2$.
Finally, section \ref{sec:For} will be devoted to explore in general the set ${\bf Forb}(\Gamma_{\leq k})$. 

Now, we obtain the characterization of $\Gamma_{\leq 1}$.

\begin{Theorem}\label{teo:gamma1}
If $G$ is a simple connected graph, then the following statements are equivalent:
\begin{enumerate}[(i)]
	\item $G \in \Gamma_{\leq 1}$,
	\item $G$ is $P_3$-free,
	\item $G$ is a complete graph.
\end{enumerate}
\end{Theorem}
\begin{proof}
$(i)\Rightarrow (ii)$ Since $\gamma(P_3)=2$,  then clearly $G$ must be $P_3$-free.

$(ii)\Rightarrow (iii)$
If $G$ is not a complete graph, then it has two vertices not adjacent, say $u$ and $v$.
Let $P$ be the smallest path between $u$ and $v$.
Thus, the length of $P$ is greater or equal to $3$. 
So, $P_3$ is an induced subgraph of $P$ and hence of $G$.
Therefore, $G$ is a complete graph.

$(iii)\Rightarrow (i)$ It is easy to see that for any non-trivial simple connected graph, its first critical ideal is trivial, meanwhile $I_1(K_1,\{ x\}) =\left< x \right>$. 
On the other hand, the 2-minors of a complete graphs are of the forms: $-1 + x_ix_j$ and $1+ x_i$.
Since $-1 + x_ix_j\in \left< 1+ x_1, ..., 1+x_n \right>$, then
\begin{equation}\label{eqn:gamma1}
	I_2(K_n,X_{K_n})=
	\begin{cases}
		\left< -1 + x_1x_2 \right> & \text{if } n=2, \text{ and,}\\
		\left< 1+ x_1, ..., 1+x_n \right> & \text{if } n\geq 3.
	\end{cases}
\end{equation}
Therefore $\gamma(K_n)\leq 1$.
In fact, the set $\{ 1+ x_1, ..., 1+x_n\}$ is a reduced Gr\"obner basis for $I_2(K_n,X_{K_n})$, see \cite[theorem 3.14]{corrval}.
\end{proof}

In light of theorem~\ref{teo:gamma1}, the characterization of $\mathcal G_1$ is as follows:
Clearly, $\mathcal G_1\subseteq \Gamma_{\leq 1}\setminus\mathcal G_{0}$.
Now, let $G\in \Gamma_{\leq 1}\setminus\{ K_1\}$, that is, $G=K_n$ with $n\geq2$ and $f_1(G)\geq 1$.
It is easy to verify from equation \ref{eqn:gamma1} that the second invariant factor of $K(G)$ is equal to
$I_2(K_n,X_{K_n})\mid _{\{x_v=n-1\, | \, v \in K_n\}}$ which is different to $1$. 

\begin{Corollary}\cite{lorenzini1991}
If $G$ is a simple connected graph with $n\geq 2$ vertices, then $f_1(G)=1$ if and only if $G$ is a complete graph.
\end{Corollary}

A crucial fact in the proof of theorem \ref{teo:gamma1} was that $P_3$ belongs to ${\bf Forb}(\Gamma_{\leq 1})$
and the fact that any other connected simple graph belonging to $\Gamma_{\geq2}$ contains $P_3$.
This leads to the following corollary.
\begin{Corollary}
${\bf Forb}(\Gamma_{\leq 1})=\{ P_3\}.$
\end{Corollary}

Next corollary give us the non-connected version of theorem~\ref{teo:gamma1}.

\begin{Corollary}\label{cor:gamma1}
If $G$ is a simple non-necessary connected graph, then the following statements are equivalent:
\begin{enumerate}[(i)]
\item $\gamma(G)\leq 1$,
\item $G$ is $\{P_3, 2P_2\}$-free,
\item $G$ is a disjoint union of a complete graph and a trivial graph.
\end{enumerate}
\end{Corollary}

Before to proceed with the proof of corollary~\ref{cor:gamma1} present 
a lemma that help us to calculate the critical ideal of a non connected graph. 
It may be useful to recall that the {\it product} of the ideals $I$ and $J$ of a commutative ring $R$, 
which we denote by $IJ$, is the ideal generated by all the products $ab$ where $a\in I$ and $b\in J$. 
\begin{Lemma}\cite[Proposition 3.4]{corrval}\label{lemm:corrvaldisjoint}
If $G$ and $H$ are vertex disjoint graphs, then
\[
I_i(G+H,\{X_G,Y_H\})=\left< \cup_{j=0}^{i}I_j(G,X_G) I_{i-j}(H,Y_H) \right> \text{ for all }1\leq i\leq |V(G+H)|.
\]
\end{Lemma}
By this lemma we have that $\gamma(G+H)=\gamma(G)+\gamma(H)$ when $G$ and $H$ are vertex disjoint.

\begin{proof}[Proof of corollary~\ref{cor:gamma1}]
$(i)\Rightarrow (ii)$ 
It follows since $\gamma(2P_2)=2$ and $\gamma(P_3)=2$.

$(ii)\Rightarrow (iii)$
Let $G_1,\ldots,G_s$ be the connected components of $G$.
Then by theorem~\ref{teo:gamma1} and lemma~\ref{lemm:corrvaldisjoint}, $G_i$ is a complete graph for all $1\leq i\leq s$.
Since $2P_2$ must not be an induced subgraph of $G$, then at most one of the $G_i$ has order greater than $1$.

$(iii)\Rightarrow (i)$ 
If $G = K_n + T_m$, then it is not difficult to see that $I_1(T_m,Y_{T_m})=\left< y_1, ..., y_m\right>$ 
and $I_2(T_m,Y_{T_m})=\left< \prod_{i\neq j}y_i y_j \right>$.
Thus by lemma~\ref{lemm:corrvaldisjoint}, 
\[
I_2(G,\{ X_{K_n}, Y_{T_m}\})=\left< I_2(K_n,X_{K_n}), I_1(K_n,X_{K_n})I_1(T_m,Y_{T_m}), I_2(T_m,Y_{T_m})\right>\neq\left< 1\right>. \vspace{-9mm}
\]
\end{proof}

\section{Graphs with algebraic co-rank equal to two}
\label{gamma2}
The main goal of this section is to classify the simple graphs on $\Gamma_{\leq 2}$.
After that, using the fact that $\mathcal{G}_2\subseteq \Gamma_{\leq 2}$ we will classify the simple graphs
whose critical group has two invariant factors equal to $1$.
As in the case of $\Gamma_{\leq 1}$, the characterization of $\Gamma_{\leq 2}$ relies heavely
in the fact that $\Gamma_{\leq 2}$ is closed under induced subgraphs and the fact that we have a good
guessing about ${\bf Forb}(\Gamma_{\leq 2})$.
We begin with the introduction of a set of graphs in ${\bf Forb}(\Gamma_{\leq 2})$.

\begin{Proposition}\label{lem::main1}
Let $\mathcal F_2$ be the set of graphs consisting of $P_4$, $K_5\setminus S_2$, $K_6\setminus M_2$, $\Gaa$ and $\Gab$; see figure~\ref{fig2}.
Then $\mathcal{F}_2 \subseteq {\bf Forb}(\Gamma_{\leq 2})$.
\end{Proposition}
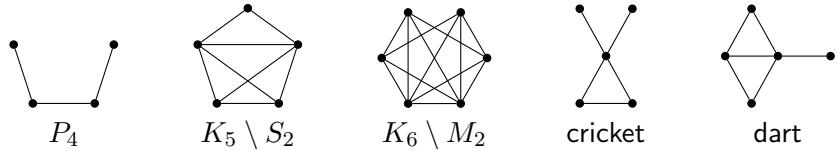
\begin{figure}[h]
\begin{center}
\begin{tabular}{c@{\extracolsep{10mm}}c@{\extracolsep{10mm}}c@{\extracolsep{10mm}}c@{\extracolsep{10mm}}c}
	\begin{tikzpicture}[scale=.7]
	\tikzstyle{every node}=[minimum width=0pt, inner sep=1pt, circle]
	\draw (126+36:1) node (v1) [draw,fill] {};
	\draw (198+36:1) node (v2) [draw,fill] {};
	\draw (270+36:1) node (v3) [draw,fill] {};
	\draw (342+36:1) node (v4) [draw,fill] {};
	\draw (v1) -- (v2);
	\draw (v2) -- (v3);
	\draw (v4) -- (v3);
	\end{tikzpicture}
&
	\begin{tikzpicture}[scale=.7]
	\tikzstyle{every node}=[minimum width=0pt, inner sep=1pt, circle]
	\draw (126-36:1) node (v1) [draw,fill] {};
	\draw (198-36:1) node (v2) [draw,fill] {};
	\draw (270-36:1) node (v3) [draw,fill] {};
	\draw (342-36:1) node (v4) [draw,fill] {};
	\draw (414-36:1) node (v5) [draw,fill] {};
	\draw (v1) -- (v2);
	\draw (v1) -- (v5);
	\draw (v2) -- (v3);
	\draw (v2) -- (v4);
	\draw (v2) -- (v5);
	\draw (v3) -- (v4);
	\draw (v3) -- (v5);
	\draw (v4) -- (v5);
	\end{tikzpicture}
&
	\begin{tikzpicture}[scale=.7]
	\tikzstyle{every node}=[minimum width=0pt, inner sep=1pt, circle]
	\draw (0:1) node (v1) [draw,fill] {};
	\draw (60:1) node (v2) [draw,fill] {};
	\draw (120:1) node (v3) [draw,fill] {};
	\draw (180:1) node (v4) [draw,fill] {};
	\draw (240:1) node (v5) [draw,fill] {};
	\draw (300:1) node (v6) [draw,fill] {};
	\draw (v1) -- (v2);
	\draw (v1) -- (v3);
	\draw (v1) -- (v5);
	\draw (v1) -- (v6);
	\draw (v2) -- (v4);
	\draw (v2) -- (v5);
	\draw (v2) -- (v6);
	\draw (v3) -- (v4);
	\draw (v3) -- (v5);
	\draw (v3) -- (v6);
	\draw (v4) -- (v5);
	\draw (v4) -- (v6);
	\draw (v5) -- (v6);
	\end{tikzpicture}
&
	\begin{tikzpicture}[scale=.7]
	\tikzstyle{every node}=[minimum width=0pt, inner sep=1pt, circle]
	\draw (-.5,-.9) node (v1) [draw,fill] {};
	\draw (.5,-.9) node (v2) [draw,fill] {};
	\draw (0,0) node (v3) [draw,fill] {};
	\draw (-.5,.9) node (v4) [draw,fill] {};
	\draw (.5,.9) node (v5) [draw,fill] {};
	\draw (v1) -- (v2);
	\draw (v1) -- (v3);
	\draw (v2) -- (v3);
	\draw (v3) -- (v4);
	\draw (v3) -- (v5);
	\end{tikzpicture}
&
	\begin{tikzpicture}[scale=.7]
	\tikzstyle{every node}=[minimum width=0pt, inner sep=1pt, circle]
	\draw (-.5,0) node (v2) [draw,fill] {};
	\draw (0,-.9) node (v1) [draw,fill] {};
	\draw (.5,0) node (v3) [draw,fill] {};
	\draw (1.5,0) node (v5) [draw,fill] {};
	\draw (0,.9) node (v4) [draw,fill] {};
	\draw (v1) -- (v2);
	\draw (v1) -- (v3);
	\draw (v2) -- (v3);
	\draw (v2) -- (v4);
	\draw (v3) -- (v4);
	\draw (v3) -- (v5);
	\end{tikzpicture}
\\
$P_4$
&
$K_5\setminus S_2$
&
$K_6\setminus M_2$
&
\Gaa
&
\Gab
\end{tabular}
\end{center}
\caption{The set $\mathcal{F}_2$ of graphs.}
\label{fig2}
\end{figure}
\begin{proof}
It is not difficult to see that the generalized Laplacian matrix of the graphs on $\mathcal{F}_2$ are given by:
\begin{center}
\begin{tabular}{l@{\extracolsep{12mm}}r}
$
\hspace{6.5mm}
L(P_4)
=
\left[
	\begin{array}{ccccc}
		  x_{1} & -1 & 0 &  0\\ 
		 \cellcolor[gray]{0.7}-1 &\cellcolor[gray]{0.7} x_{2} & \cellcolor[gray]{0.7}-1 &  0\\ 
		  \cellcolor[gray]{0.7}0 & \cellcolor[gray]{0.7}-1 & \cellcolor[gray]{0.7}x_{3} &  -1\\ 
		  \cellcolor[gray]{0.7}0 & \cellcolor[gray]{0.7} 0 & \cellcolor[gray]{0.7}-1 & x_{4} \\ 
	\end{array}
\right],
$
&
$
L(K_5\setminus S_2)
=
\left[
	\begin{array}{ccccc}
		 \cellcolor[gray]{0.7} x_{1} & 0 & -1 & \cellcolor[gray]{0.7} -1 & \cellcolor[gray]{0.7} 0 \\ 
		 \cellcolor[gray]{0.7} 0 & x_{2} & -1 & \cellcolor[gray]{0.7}  -1& \cellcolor[gray]{0.7} -1 \\ 
		 \cellcolor[gray]{0.7} -1 & -1 & x_{3} & \cellcolor[gray]{0.7} -1& \cellcolor[gray]{0.7} -1 \\ 
		 -1 & -1 & -1 & x_{4} & -1 \\ 
		 0 & -1 & -1 & -1 & x_{5} \\
	\end{array}
\right],
$
\end{tabular}
\\
\begin{tabular}{l@{\extracolsep{1cm}}r}
$
L(\Gaa)
=
\left[
	\begin{array}{ccccc}
		  x_{1} & 0 & \cellcolor[gray]{0.7} 0 & \cellcolor[gray]{0.7} -1 & \cellcolor[gray]{0.7} 0 \\ 
		 0 & x_{2} & \cellcolor[gray]{0.7} -1 & \cellcolor[gray]{0.7} -1 & \cellcolor[gray]{0.7} 0 \\ 
		 0 & -1 & x_{3} & -1 & 0 \\ 
		 -1 & -1 & \cellcolor[gray]{0.7} -1 & \cellcolor[gray]{0.7} x_{4} & \cellcolor[gray]{0.7} -1 \\ 
		 0 & 0 & 0 & -1 & x_{5} \\ 
\end{array}
\right],
$
&
$
L(\Gab)
=
\left[
\begin{array}{ccccc}
 x_{1} & -1 & \cellcolor[gray]{0.7} 0 & \cellcolor[gray]{0.7} -1 & \cellcolor[gray]{0.7} 0 \\ 
 -1 & x_{2} & \cellcolor[gray]{0.7} -1 & \cellcolor[gray]{0.7} -1 & \cellcolor[gray]{0.7} 0 \\ 
 0 & -1 & x_{3} & -1 & 0 \\ 
 -1 & -1 & \cellcolor[gray]{0.7} -1 & \cellcolor[gray]{0.7} x_{4} & \cellcolor[gray]{0.7} -1 \\ 
 0 & 0 & 0 & -1 & x_{5} \\ 
\end{array}
\right],
$
\end{tabular}
\\
\begin{tabular}{l@{\extracolsep{1cm}}r}
\end{tabular}
\begin{tabular}{c}
$
L(K_6\setminus M_2)
=
\left[
	\begin{array}{cccccc}
		  x_{1} & 0 & -1 &  -1 & -1 & -1\\ 
		 \cellcolor[gray]{0.7} 0 & x_{2} & -1 &  -1&  \cellcolor[gray]{0.7}-1 &\cellcolor[gray]{0.7} -1\\ 
		 \cellcolor[gray]{0.7}-1 & -1 & x_{3} &  -1& \cellcolor[gray]{0.7} -1 & \cellcolor[gray]{0.7}0\\ 
		 \cellcolor[gray]{0.7}-1 & -1 & -1 & x_{4} & \cellcolor[gray]{0.7}-1  & \cellcolor[gray]{0.7}-1\\ 
		 -1 & -1 & -1 & -1 & x_{5}  & -1\\
		 -1 & -1 &  0 & -1 & -1 & x_{6} \\
	\end{array}
\right].
$
\end{tabular}
\end{center}
In this matrices we marked with gray some $3\times 3$ square submatrices whose determinant is equal to $\pm 1$.
Then $\gamma(G)\geq 3$ for all $G\in \mathcal{F}_2$.
Finally, using any algebraic system, for instance {\it Macaulay 2}, one can note that the graphs in $\mathcal{F}_2$
has algebraic co-rank equal to $3$.
Moreover, it can be checked that any of his induced subgraphs has algebraic co-rank less or equal to $2$.
\end{proof}

One of the main results of this article is the following:

\begin{Theorem}\label{teo:main1}
Let $G$ be a simple connected graph.
Then, $G\in \Gamma_{\leq 2}$ if and only if $G$ is an induced subgraph of $K_{m,n,o}$ or $T_n\vee (K_m+K_o)$.
\end{Theorem}

We divide the proof of theorem~\ref{teo:main1} in two steps. 
First we classify the connected graphs that are $\mathcal{F}_2$-free.
After that, we check that all these graphs have algebraic co-rank less or equal than two.

\begin{Theorem}\label{teo:main1.1}
A simple connected graph is $\mathcal{F}_2$-free if and only if it is an induced subgraph of $K_{m,n,o}$ or $T_n\vee (K_m+K_o)$.
\end{Theorem}
\begin{proof}
First, one implication is clear because $K_{m,n,o}$ and $T_n\vee (K_m+K_o)$ are $\mathcal{F}_2$-free.
The another part of the proof is divided in three cases: when $\omega(G)=2$, $\omega(G)=3$, and $\omega(G)\geq 4$.

\medskip

The case when $\omega(G)=2$ is very simple.
Since $\omega(G)=2$, there exist $a,b\in V(G)$ such that $ab\in E(G)$.
Clearly, $N_G(a)\cap N_G(b)=\emptyset$.
Moreover, if $x\in \{a,b\}$, then $uv\notin E(G)$ 
for all $u,v\in N_G(x)$.
On the other hand, since $G$ is $P_4$-free, then $uv\in E(G)$ for all $u\in N_G(a)$ and $v\in N_G(b)$.
Therefore $G$ is the complete bipartite graph.

\medskip

Now, assume that $\omega(G)= 3$.
Let $a$, $b$ and $c$ be vertices of $G$ that induce a complete graph.
For all $X\subseteq \{a, b, c\}$ let $V_X=\{v\in V(G)\, | \, N_G(v)\cap\{a, b, c\}=X \}$.
Clearly $V_{\{a, b, c\}}=\emptyset$ because $\omega(G)=3$.
In a similar way, if $X\subseteq \{a, b, c\}$ has size two, then set $V_X$ induce a trivial graph.
Also, since $G$ is $\Gaa$-free, $V_x$ induces a complete graph for all $x\in \{a, b, c\}$.
Thus $V_x$ has at most two vertices.

Now, given $U,V \in V(G)$, let $E(U,V)=\{uv\in E(G)\, | \, u\in U \text{ and } v\in V\}$.
Let $x\neq y\in \{a,b,c\}$ and $z\in \{a,b,c\}$ such that $\{x,y,z\}=\{a,b,c\}$
Assume that $V_x,V_y$ and $V_{\{x,y\}}$ are not empty.
Let $u\in V_x$ and $v\in V_y$.
If $uv\in E(G)$, then $\{u,v,y,z\}$ induced a $P_4$.
Therefore $E(V_x,V_y)=\emptyset$.
In a similar way, since $G$ is $P_4$-free, we get $E(V_x,V_{\{x,y\}})=\emptyset$.
\begin{Claim}
At least two of the sets $V_a$, $V_b$ or $V_c$ are empty.
Furthermore, if $V_a\neq \emptyset$, then $G$ is an induced subgraph of 
$T_l\vee (K_2+K_2)$, where $l=|V_{\{b,c\}}|+1$. 
\end{Claim}
\begin{proof}
First, assume that $V_x$ and $V_y$ are non empty.
Let $u\in V_y$, $v\in V_x$.
Since $u$ and $v$ are not adjacent, the vertices $\{u,x,y,v\}$ induces a $P_4$.
Therefore at least one of $V_x$ or $V_y$  is empty.

Without loss of generality we can assume that $V_a$ is not empty.
Since there is no edge between $V_a$ and $V_{\{a,b\}}$, then $V_{\{a,b\}}=\emptyset$.
Otherwise, if $u\in V_{\{a,b\}}$ and $v\in V_a$, then the vertices $\{u,v,a,b,c\}$ induces a $\Gab$.
In a similar way $V_{\{a,c\}}=\emptyset$.
On the other hand, if $V_{\{b,c\}}$ is not empty and there exist $u\in V_{\{b,c\}}$ and $v\in V_a$ such that $uv\notin E(G)$, 
then the vertices $\{u,b,a,v\}$ induces a $P_4$.
Therefore, either $E(V_a,V_{\{b,c\}})=\{uv\, |\, u\in V_a \text{ and } v\in V_{\{b,c\}}\}$ or the set $V_{\{b,c\}}$ is empty.
Finally, since $V_a$ is a complete graph with at most two vertices, the result follows.
\end{proof} 
Now, we can assume that $V_x=\emptyset$ for all $x\in\{ a, b, c\}$.
Let $\{x,y,z\}= \{ a,b,c\}$. If $uv\notin E(G)$ for some $u\in V_{\{x,y\}}$, $v\in V_{\{x,z\}}$, then $\{u,y,z,v\}$ induces a $P_4$.
Therefore $uv\in E(G)$ for all $u\in V_{\{x,y\}}$ and $u\in V_{\{x,z\}}$, and $G$ is an induced subgraph of the complete tripartite graph.

\medskip

We finish with case when $\omega(G)\geq 4$.
Let $W=\{a,b,c,d\}$ be a complete subgraph of $G$ of size four and  let
\[
V_i=\{v\in V(G)\setminus W\, \big| \, |N_G(V)\cap W|=i \} \text{ for all } i=0,1,2,3,4.
\]
Since $G$ is $K_5\setminus S_2$-free, $V_2=\emptyset$.

\begin{Claim}
The graph induced by $V_1$ is a complete graph.
\end{Claim}
\begin{proof}
Let $u,u'\in V_{1}$ and suppose there is no edge between $u$ and $u'$.
Let $x,y\in W$ be the vertices adjacent to $u$ and $u'$, respectively. 
If $x\neq y$, then $\{u,x,y,u'\}$ induces a $P_4$; a contradiction.
On the other hand, if $x=y$, let $z\neq w\in W\setminus x$.
Since $u$ and $u'$ are not adjacent to both $z$ and $w$,
then $\{x,z,w,u,u'\}$ induces a $\Gaa$; a contradiction.
\end{proof}
Let $v,v'\in V_{3}$ and assume that are adjacent.
Let $x,y\in W$  such that $x\notin N_G(v)$ and $y\notin N_G(v')$.
If $x\neq y$, then $\{v,v'\}\cup W$ induces a $K_6\setminus M_2$; a contradiction.
On the other hand, if $x=y$, then $\{v,v'\}\cup W$ contains a $K_5\setminus S_2$ as induced graph; a contradiction.
Therefore $V_{3}$ induces a trivial graph.

Now, let $u\in V_1$, $v\in V_3$, $x, y\in W$ such that $xu\in E(G)$, $yv\notin E(G)$.
Assume that $uv\notin E(G)$.
Let $z\in W\setminus \{x,y\}$. 
If $x=y$, then $\{v,z,x,u\}$ induces a $P_4$; a contradiction.
On the other hand, if $x\neq y$, then $G$ must contains a $\Gab$ as induced subgraph; a contradiction.
Therefore $E(V_{1},V_3)$ contains all the possible edges.
Since $uv\in E(G)$, then $x=y$.
Otherwise, if $x\neq y$, then $\{y,z,v,u\}$ induces a $P_4$; a contradiction.
Therefore we can assume without loss of generality that $\{a\}=N_G(V_1)\cap W=(N_G(V_3)\cap W)^c$.

Now, let $w\in V_4$, $u\in V_1$, and $v\in V_3$.
If $uw\in E(G)$, then $\{u,w,a,b,c\}$ induces a $K_5\setminus S_2$.
Therefore, $E(V_1,V_4)=\emptyset$.
In a similar way, if $vw\notin E(G)$, then $\{v,a,w,b,c\}$ induces a $K_5\setminus S_2$.
Therefore, $E(V_3,V_4)=\{vw\, | \, v\in V_3 \text{ and } w\in V_4\}$.

Since $G$ is $\{K_5\setminus S_2,K_6\setminus M_2\}$-free, then it is not difficult to see that
the graph induced by $V_4$ is $\{K_2+T_1,C_4\}$-free.
Thus $V_4$ induces either a trivial graph, a complete graph, or a complete graph minus an edge.
Moreover, if $ww'\notin E(G)$ for some $w\neq w'\in V_4$ and $v\in V_3$, then $\{w,w',a,v,b,c\}$ induces a $K_6\setminus M_2$.
Thus, if $V_3\neq \emptyset$, then $V_4$ induces a complete graph.

Clearly, if $V_1,V_3,V_4=\emptyset$, then $G$ is a complete graph.
\begin{Claim}
If $V_1,V_3=\emptyset$ and $V_4\neq \emptyset$, then $G$ is an induced subgraph of $T_1\vee (K_m+K_n)$ for some $m,n\in \mathbb{N}$.
\end{Claim}
\begin{proof}
If $|V_0|=|V_4|=1$, then the result is clear.
Therefore we can assume that $|V_4|\geq 2$ or $|V_0|\geq 2$.
Moreover we need to consider three cases for $V_4$, when it induces a trivial graph, a complete graph, or a complete graph minus an edge.
Assume that $V_4$ induces a trivial graph.
If $|V_4|\geq 2$, let $o \in V_0$ and $w,w'\in V_4$.
If $ow\in E(G)$ and $ow'\notin E(G)$, then $\{o,w,w',a\}$ induces a $P_4$; a contradiction.
Thus, either $E(o, V_4)=\{ow\, | \, w\in V_4\}$ or $E(o, V_4)$ is empty.
Therefore, since $G$ is connected, we get the result when $|V_0|=1$.

Now, assume that $|V_0|\geq 2$. 
Since $G$ is connected, there exist $o\in V_0$ such that $ow\in E(G)$ for some $w\in V_4$.
Let $o'\in V_0$ such that $E(o', V_4)$ is empty.
Since $G$ is connected, there exist a path from $o'$ to $o$.
Let $P$ be a minimum path between $o'$ and $o$.
In this case, $\{V(P),w,a\}$ induces a path with more that four vertices; a contradiction.
Therefore, $E(V_0, V_4)=\{ow\, | \, o\in V_0 \text{ and }w\in V_4\}$.
Moreover, since $G$ is $K_6\setminus M_2$-free, then $V_0$ induces a trivial graph and we get the result.

Now, assume that $V_4$ induces a complete graph.
Since $G$ is $K_5\setminus S_2$-free, $o$ is adjacent to at most one vertex in $V_4$.
Moreover, all the vertices in $V_0$ are adjacent to a unique vertex in $V_4$.
Otherwise, let $o,o'\in V_0$ and $w,w'\in V_4$ such that $ow,o'w'\in E(G)$ and $ow',o'w\notin E(G)$.
If $oo'\in E(G)$, then $\{a,w,o,o'\}$ induces a $P_4$; a contradiction.
Also, if $oo'\notin E(G)$ and $ww'\in E(G)$, then $\{w,w',o,o'\}$ induces a $P_4$; a contradiction.
Let $w\in V_4$ such that all the vertices in $V_0$ are adjacent to $w$.
Then $V_0$ induces a complete graph.
Otherwise, $\{a,b,w,o,o'\}$ induces a $\Gaa$; a contradiction.
Therefore $G$ is an induced subgraph of $T_1\vee (K_m+K_n)$ for some $m,n\in \mathbb{N}$.

Finally, when $V_4$ induces a complete graph minus an edge, following similar arguments
to those given in the case when $V_4$ induces a complete graph we get that $G$ is an induced 
subgraph of $T_2\vee (K_m+K_n)$ for some $m,n\in \mathbb{N}$.
\end{proof}

Therefore we can assume that $V_1\cup V_{3}\neq \emptyset$. 
Let $u\in V_1\cup V_{3}$, $o\in V_0$,  and $x\neq y \in W$ such that $x\notin N_G(u)$ and $y\in N_G(u)$.
If $uo\in E(G)$, then $\{x,y,u,o\}$ induces a $P_4$; a contradiction.
Thus, there are no edges between $V_0$ and $V_1\cup V_3$.
Moreover, let $w\in V_4$. 
If $ow\in E(G)$, then $\{a,b,u,w,o\}$ induces a $\Gab$ when $u\in V_3$ and $\{u,a,w,o\}$ induces a $P_4$ when $u\in V_1$.
Therefore there are no edges between $V_0$ and $V_4$.
Since $G$ is connected, $V_0=\emptyset$ and therefore $G$ is an induced subgraph of $T_n\vee (K_m+K_o)$.
\end{proof}

To finish the proof of theorem~\ref{teo:main1} we need to prove that the third critical 
ideal of the graphs $K_{m,n,o}$ and $T_n\vee (K_m+K_o)$ is not trivial.
If $m+n+o\leq 2$, then the third critical ideal is equal to zero.
Also, if $m+n+o= 3$, then the third critical ideal is equal to the determinant of the generalized Laplacian matrix.
Moreover, \cite[theorem 3.16]{corrval} 
proves that the algebraic co-rank of the complete graph is equal to $1$.

\begin{Theorem}\label{teo:main1.2}
If $K_{m,n,o}$ is connected with $m\geq n \geq o$ and $m+n+o\geq 4$, then
\begin{equation}\label{I3K}
\small
I_3(K_{m,n,o},\{X,Y,Z\})=
\begin{cases}
	\left< 2, \bigcup_{i=1}^{m}x_i, \bigcup_{i=1}^{n}y_i, \bigcup_{i=1}^{o}z_i \right> & \text{if } m, n, o\geq 2,\\
	\left< \bigcup_{i=1}^{m}x_i, \bigcup_{i=1}^{n}y_i, z_1+2\right> & \text{if } m\geq 2, n\geq 2, o=1,\\
	\left< \bigcup_{i=1}^{m}x_i, y_1+z_1+2 \right> & \text{if } m\geq 3, n=1, o=1,\\
	\left< x_1x_2+x_1+x_2, x_1 z_1+x_1, x_2 z_1+x_2, y_1+z_1+2\right> & \text{if } m=2, n=1, o=1,\\
	\left< \bigcup_{i=1}^{m}x_i, \bigcup_{i=1}^{n}y_i \right> & \text{if } m\geq 3, n\geq 3, o=0,\\
	\left< \bigcup_{i=1}^{m}x_i, y_1+y_2 \right> & \text{if } m\geq 3, n=2, o=0,\\
	\left< x_2y_2,x_1+x_2,y_1+y_2\right> & \text{if } m=2, n=2, o=0,\\
	\left< \bigcup_{i=1}^{m}x_i \right> & \text{if } m\geq 3, n=1, o=0.
\end{cases}
\end{equation}
\end{Theorem}

\begin{Theorem}\label{teo:main1.3}
If $T_n\vee (K_m+K_o)$ is connected with $m \geq o$, $m+n+o\geq 4$, 
and such that $T_n\vee (K_m+K_o)$ is not the complete graph or the complete bipartite graph, then
\begin{equation}\label{I3L}
\small
I_3(T_n\vee (K_m+K_o),\{X,Y,Z\})=
\begin{cases}
\left< 2, \bigcup_{i=1}^{m}(x_i+1), \bigcup_{i=1}^{n}y_i, \bigcup_{i=1}^{o}(z_i+1) \right> & \text{if } m, n, o\geq 2,\\
\left< \bigcup_{i=1}^{m}(x_i+1), y_1+2,\bigcup_{i=1}^{o}(z_i+1) \right> & \text{if } m\geq 2, n=1, o\geq 2,\\
\left< \bigcup_{i=1}^{m}(x_i+1), \bigcup_{i=1}^{n}y_i, z_1-1 \right> & \text{if } m\geq 2,  n \geq 2,  o=1,\\
\left< x_1+z_1, \bigcup_{i=1}^{n}y_i \right> & \text{if } m=1,  n\geq 3,  o=1,\\
\left< x_1+z_1, y_1+y_2, y_2 z_1,\right> & \text{if }  m=1,  n = 2,  o=1,\\
\left<  \bigcup_{i=1}^{m}(x_i+1), z_1 y_1+z_1-1\right> & \text{if } m\geq 2,  n=1,  o=1,\\
\left< \bigcup_{i=1}^{m}(x_i+1), \bigcup_{i=1}^{n}y_i \right> & \text{if } m\geq 3,  n \geq 3,  o=0,\\
\left< x_1+x_2+2, \bigcup_{i=1}^{n}y_i \right> & \text{if } m=2,  n\geq 3,  o=0,\\
\left< \bigcup_{i=1}^{m}(x_i+1), y_1 y_2+y_1+y_2 \right> & \text{if } m\geq 3,  n=2,  o=0,\\
\left< x_1+x_2+2, x_2 y_1+y_1, x_2 y_2+y_2, y_1 y_2+y_1+y_2\right> & \text{if } m=2,  n =2,  o=0,\\
\end{cases}
\end{equation}
\end{Theorem}

The proofs of theorems~\ref{teo:main1.2} and~\ref{teo:main1.3} relies on the description of the 
$3$-minors of the generalized Laplacian matrices of $K_{m,n,o}$ and $T_n\vee (K_m+K_o)$.


\begin{proof}[Proof of theorem~\ref{teo:main1.2}] 
In order to simplify the arguments in the proof we separate it in two parts.
We begin by finding the $3$-minors of the generalized Laplacian matrix of the complete bipartite graph
and using it to calculate their third critical ideal.
An after that, we do the same for the general case of the complete tripartite graph.

\begin{Lemma}\label{lemma:minorsKmn}
For $m,n\geq 1$, let $L_{m,n}$ be the generalized Laplacian matrix of the complete bipartite graph $K_{m,n}$.
That is, 
\[
L_{m,n}=L(K_{m,n},\{ X_{T_m}, Y_{T_n}\})=\left[
	\begin{array}{ccccccccc}
		L(T_{m},X_{T_{m}}) & -J_{m, n}\\
		-J_{n, m} & L(T_{n},Y_{T_{n}})\\
	\end{array}
\right].
\]
Then $3$-minors (with positive leading coefficient) of $L_{m,n}$ are the following:

{\centering
\small
\begin{tabular}{l@{\extracolsep{20mm}}l}
$\bullet$ $y_{j_1}$, $y_{j_1} y_{j_2}$, and $y_{j_1} y_{j_2} y_{j_3}$  when $n\geq 3$, & 
$\bullet$ $x_{i_1}$, $x_{i_1} x_{i_2}$, and $x_{i_1} x_{i_2} x_{i_3}$ when $m\geq 3$, \\
$\bullet$ $y_{j_1} y_{j_2} x_{i_1} - y_{j_1} - y_{j_2}$  when $n\geq 2$, &
$\bullet$ $x_{i_1} x_{i_2} y_{j_1} - x_{i_1} - x_{i_2}$  when $m\geq 2$,\\ 
$\bullet$ $x_{i_1} +x_{i_2}$, $y_{j_1} +y_{j_2}$ and $x_{i_1} y_{j_1}$  when $m\geq 2$ and $n\geq 2$,\\
\end{tabular}}

\noindent where $1\leq i_1< i_2 < i_3\leq n$ and $1\leq j_1< j_2 <  j_3\leq n$.
\end{Lemma}
\begin{proof}
Before to proceed with the proof we establish some notation corresponding to row and column indices.
Let $\mathcal{I}=\{i_1, i_2, i_3\}$ such that $1\leq i_1< i_2< i_3\leq m+n$, and $\mathcal{J}=\{j_1,j_2,j_3\}$ such that $1\leq j_1<j_2<j_3\leq m+n$.
Let $\mathcal{I}_1=\mathcal{I}\cap [m]$, $\mathcal{I}_2=\mathcal{I}^c_1$, $\mathcal{J}_1=\mathcal{J}\cap [m]$, and $\mathcal{J}_2=\mathcal{J}^c_1$.
Also in the following $i'_t=i_t-m$ and $j'_t=j_t-m$, for all $1\leq t \leq 3$.

In order to find all the $3$-minors of $L_{m,n}$ we need to calculate the determinants of all non-singular matrices of the form $L_{m,n}[\mathcal{I},\mathcal{J}]$.
Since the generalized Laplacian matrix is symmetric, we can assume without loss of generalization that $|\mathcal{I}_2| \leq |\mathcal{J}_2|$.
Let $L=L_{m,n}[\mathcal{I};\mathcal{J}]$ be non-singular.
First, consider the case when $\mathcal{I}_2$ is empty. 
Since the determinant of $L$ is equal to zero when $|\mathcal{J}_2|\geq 2$, only remains to consider the cases when $|\mathcal{J}_2|=0$ or $|\mathcal{J}_2|=1$.
If $|\mathcal{J}_2|=0$, then $m\geq 3$, $L$ is a submatrix of $L(T_{m},X_{T_{m}})$, and the determinant of $L$ is equal to $x_{i_1}x_{i_2}x_{i_3}$.
In a similar way, if $|\mathcal{J}_2|=1$, then $m\geq 3$, $n\geq 1$, and $L$ is equal to (up to row permutation) 
\[
\left[
	\begin{array}{ccc}
		x_{j_1} & 0 & -1\\
		0 & x_{j_2} & -1\\
		0 & 0 & -1\\
	\end{array}
\right] 
\]
whose determinant is equal to $-x_{j_1}x_{j_2}$.

Now, consider the case when $|\mathcal{I}_2|=1$.
In a similar way, $L$ has determinant different from zero when $|\mathcal{J}_2|=1$ or $|\mathcal{J}_2|=2$.
If $|\mathcal{J}_2|=1$, then there are essentially only four $3\times 3$ non-singular submatrices of $L_{m,n}$:
\[
\left[
	\begin{array}{ccccccccc}
		x_{i_1} & 0 & -1\\
		0 & A & -1\\
		-1 & -1 & B\\
	\end{array}
\right], 
\]
where $A$ is equal to $0$ (when $m\geq 3$) and $x_{i_2}$, and $B$ is equal to $0$ (when $n\geq 2$) and $y_{i'_3}$.
Clearly $\det(L)=ABx_{i_1}-A-x_{i_1}$.
Thus we have the following minors:
$x_{i_1}x_{i_2}y_{i'_3}-x_{i_1}-x_{i_2}$, $-x_{i_1}-x_{i_2}$, $-x_{i_1}$.
If $|\mathcal{J}_2|=2$, then $m\geq 2$, $n\geq 2$, and $L$ has determinant equal to
\[
\det\left[
	\begin{array}{ccccccccc}
		x_{j_1} & -1 & -1\\
		0 & -1 & -1\\
		-1 & 0 & y_{i'_3}\\
	\end{array}
\right]=-x_{j_1}y_{i'_3}.
\]
When $|\mathcal{I}_2|=2$ we have two cases, when either $|\mathcal{J}_2|=2$ or $|\mathcal{J}_2|=3$.
If $|\mathcal{J}_2|=2$, then $L$ is equal to: 
\[
\left[
	\begin{array}{ccccccccc}
		A & -1 & -1\\
		-1 & y_{i'_2} & 0\\
		-1 & 0 & B\\
	\end{array}
\right]
\]
where $A$ is equal to $0$ (when $m\geq 2$) or $x_{i_1}$ and $B$ is equal to $0$ (when $n\geq 3$) or $y_{i'_3}$.
It is easy to see that $\det(L)=ABy_{i'_2}-A-y_{i'_2}$.
Thus we have the following minors:
$x_{i_1}y_{i'_2}y_{i'_3}-y_{i'_2}-y_{i'_3}$, $-y_{i'_2}-y_{i'_3}$, $-y_{i'_2}$.
If $|\mathcal{J}_2|=3$, then $m\geq 1$, $n\geq 3$ and there are only one non-singular matrix whose determinant is equal to
\[
\det \left[
	\begin{array}{ccccccccc}
		-1 & -1 & -1\\
		y_{i'_2} & 0 & 0\\
		0 & y_{i'_3} & 0\\
	\end{array}
\right]=-y_{i'_2}y_{i'_3}.
\]
Finally, if $|\mathcal{I}_2|=3$, then $n\geq 3$, $L$ is a submatrix of $L(T_{m},Y_{T_{m}})$, and therefore its determinant is equal to $y_{i'_1}y_{i'_2}y_{i'_3}$.
\end{proof}

We can use lemma~\ref{lemma:minorsKmn} to get the third critical ideal of the complete bipartite graph.
For instance, it is not difficult to see that $I_3(K_{m,n},\{X,Y\})=\left< \bigcup_{i=1}^{m}x_i, \bigcup_{i=1}^{n}y_i \right>$ when $m\geq 3$ and $n\geq 3$.
In a similar way, since $x_{i_1} +x_{i_2}, x_{i_1} y_{j_1}, y_{j_1} y_{j_2} x_{i_1} - y_{j_1} - y_{j_2}, x_{i_1} x_{i_2},x_{i_1} x_{i_2} x_{i_3}
\in \left< \bigcup_{i=1}^{m}x_i, y_1+y_2 \right>$, $I_3(K_{m,n},\{X,Y\})=\left< \bigcup_{i=1}^{m}x_i, y_1+y_2 \right>$ when $m\geq 3$ and $n=2$.
The other cases follow in a similar way.

Therefore in order to calculate the third critical ideal of the complete tripartite graph we need to calculate their $3$-minors as below.

\begin{Theorem}\label{teo:minorsKmno}
For $m,n,o\geq 1$, let $L_{m,n,o}$ be the generalized Laplacian matrix of the tripartite complete graph $K_{m,n,o}$.
That is,
\[
L_{m,n,o}=L(K_{m, n, o},\{X_{T_{m}},Y_{T_{n}},Z_{T_{o}}\})=
\left[
\begin{array}{ccccccccc}
	L(T_{m},X_{T_{m}}) & -J_{m, n} & -J_{m, o}\\
	-J_{n, m} & L(T_{n},Y_{T_{n}}) & -J_{n, o}\\
	-J_{o, m} & -J_{o, n} & L(T_{o},Z_{T_{o}})
\end{array}
\right].
\]
Then the $3$-minors (with positive leading coefficient) of $L_{m,n,o}$ are the following:

{\centering
\small
\begin{tabular}{l@{\extracolsep{40mm}}l}
$\bullet$ $x_{i_1}$, $x_{i_1} x_{i_2}$, and $x_{i_1} x_{i_2} x_{i_3}$ when $m\geq 3$, &
$\bullet$ $2$ when $m\geq 2$, $n\geq 2$ and $o\geq 2$,\\
$\bullet$ $y_{j_1}$, $y_{j_1} y_{j_2}$, and $y_{j_1} y_{j_2} y_{j_3}$  when $n\geq 3$, &
$\bullet$ $-2-x_{i}-y_{j}-z_{k}+x_{i} y_{j} z_{k}$,\\ 
$\bullet$ $z_{k_1}$, $z_{k_1} z_{k_2}$, and $z_{k_1} z_{k_2} z_{k_3}$ when $o\geq 3$, \\ 
\multicolumn{2}{l}{$\bullet$ $x_{i_1}$, $y_{j_1}$, $x_{i_1}+2$, $y_{j_1}+2$, $x_{i_1}+x_{i_2}$, $y_{j_1}+y_{j_2}$, and $x_{i_1} y_{j_1}$ when $ m\geq 2$ and  $n\geq 2$,} \\
\multicolumn{2}{l}{$\bullet$ $x_{i_1}$, $z_{k_1}$, $x_{i_1}+2$, $z_{k_1}+2$,$x_{i_1}+x_{i_2}$, $z_{k_1}+z_{k_2}$, and $x_{i_1}z_{k_1} $ when $m\geq 2$ and $o\geq 2$,} \\
\multicolumn{2}{l}{$\bullet$ $y_{j_1}$, $z_{k_1}$, $y_{j_1}+2$, $z_{k_1}+2$, $y_{j_1}+y_{j_2}$, $z_{k_1}+z_{k_2}$, and $y_{j_1}z_{k_1}$, when $n\geq 2$ and $o\geq 2$,} \\
\multicolumn{2}{l}{$\bullet$ {\scriptsize $y_{j_1}+z_{k_1}+2$, $x_{i_1}(y_{j_1}+1)$, $x_{i_1}(z_{k_1}+1)$, 
$x_{i_1}x_{i_2}+x_{i_1}+x_{i_2}$, $x_{i_1} x_{i_2} y_{j_1}-x_{i_1}-x_{i_2}$, and $x_{i_1} x_{i_2} z_{k_1}-x_{i_1}-x_{i_2}$ when $m\geq 2$,}} \\ 
\multicolumn{2}{l}{ $\bullet$ {\scriptsize $x_{i_1}+z_{k_1}+2$, $y_{j_1}(x_{i_1}+1)$, $y_{j_1}(z_{k_1}+1)$, $y_{j_1}y_{j_2}+y_{j_1}+y_{j_2}$,  
$y_{j_1} y_{j_2} x_{i_1}-y_{j_1}-y_{j_2}$, and $y_{j_1} y_{j_2} z_{k_1}-y_{j_1}-y_{j_2}$ when $n\geq 2$,}}\\ 
\multicolumn{2}{l}{ $\bullet$ {\scriptsize $x_{i_1}+y_{j_1}+2$, $z_{k_1}(x_{i_1}+1)$, $z_{k_1}(y_{j_1}+1)$, $z_{k_1}z_{k_2}+z_{k_1}+z_{k_2}$, 
$z_{k_1} z_{k_2} x_{i_1}-z_{k_1}-z_{k_2}$, and $z_{k_1} z_{k_2} y_{j_1}-z_{k_1}-z_{k_2}$ when $o\geq 2$,}} \\ 
\end{tabular}}

\noindent where $1\leq i_1< i_2< i_3\leq  m$, $1\leq j_1<j_2<j_3\leq n$, and $1\leq k_1< k_2 <k_3\leq o$.
\end{Theorem}
\begin{proof}
We will follow a similar proof to the proof given for lemma~\ref{lemma:minorsKmn}.
Let $\mathcal{I}=\{i_1, i_2, i_3\}$ with $1\leq i_1< i_2< i_3\leq m+n+o$ and $\mathcal{J}=\{j_1,j_2,j_3\}$ with $1\leq j_1<j_2<j_3\leq m+n+o$.
Moreover, let $\mathcal{I}_1=\mathcal{I}\cap [m]$, $\mathcal{I}_2=\mathcal{I}\cap \{m+1,\ldots,m+n\}$, $\mathcal{I}_3=\mathcal{I}\cap \{m+n+1,\ldots,m+n+o\}$, 
$\mathcal{J}_1=\mathcal{J}\cap [m]$, $\mathcal{J}_2=\mathcal{J}\cap \{m+1,\ldots,m+n\}$, $\mathcal{J}_3=\mathcal{J}\cap \{m+n+1,\ldots,m+n+o\}$.
Also, in the following $i'_t=i_t-m$, $i''_t=i_t-m-n$, $j'_t=j_t-m$ and $j''_t=j_t-m-n$, for $t\in[3]$.

Let $L=L_{m,n,o}[\mathcal{I};\mathcal{J}]$.
First, in the same way that in the proof of lemma~\ref{lemma:minorsKmn} we can assume that $L$ is non-singular.
Several of the $3$-minor of $L_{m,n,o}$ can be calculated using lemma~\ref{lemma:minorsKmn}.
For instance, if $\mathcal{I}_i=\mathcal{J}_i=\emptyset$ for some $i=1,2,3$, then $L$ is a submatrix of  $L(K_{m,n},\{ X_{T_m}, Y_{T_n}\})$
and the corresponding $3$-minor can be calculated using lemma~\ref{lemma:minorsKmn}.
Therefore we can assume that, if $\mathcal{I}_i=\emptyset$, then $\mathcal{J}_i\neq \emptyset$ for all $i=1,2,3$.
Moreover, if $\mathcal{I}_i=\emptyset$, then $|\mathcal{J}_i|=1$ for all $i=1,2,3$. 
Because otherwise either $L$ will have two identical columns; a contradiction to the fact that $L$ is non-singular. 
In a similar way, if $\mathcal{J}_i=\emptyset$, then $|\mathcal{I}_i|=1$ for all $i=1,2,3$.
If $|\mathcal{I}_i|=3$ for some $i=1,2,3$, then $L$ is a  submatrix of the generalized Laplacian matrix of a complete bipartite graph.
Therefore we can assume that $|\mathcal{I}_i|\leq 2$ and $|\mathcal{J}_i|\leq 2$ for all $i=1,2,3$.

The first case that we need to consider is when $\mathcal{I}_i\neq \emptyset \neq \mathcal{J}_i$ for all $1\leq i\leq 3$, that is, $|\mathcal{I}_i|= |\mathcal{J}_i|=1$ for all $1\leq i\leq 3$.
In this case we have that
\[ 
L=\left[
	\begin{array}{ccccccccc}
		A & -1 & -1\\
		-1 & B & -1\\
		-1 & -1 & C\\
	\end{array}
\right],
\]
where $A$ is equal to $0$ (when $m\geq 2$) or $x_{i_1}$, $B$ is equal to $0$ (when $n\geq 2$) or $y_{i'_2}$, 
and $C$ is equal to $0$ (when $n\geq 2$) or $z_{i''_3}$.
Since $\det L=ABC-A-B-C-2$ we get eight of the $3$-minors of $L_{m,n,o}$.  
Since $|\mathcal{I}_i| \leq 2$ ($|\mathcal{J}_i| \leq 2$) for all $i=1,2,3$, then there are no two $\mathcal{I}$'s ($\mathcal{J}$'s) empty.
Therefore only remains the cases: when only one of the $\mathcal{I}$'s is empty and the case when one of the $\mathcal{I}$'s is empty and one of the $\mathcal{J}'s$ is empty.

Consider the case when only one of the sets $\mathcal{I}$'s is empty, that is, $|\mathcal{J}_i|=1$ for all $i=1,2,3$.
Assume that $\mathcal{I}_3=\emptyset$. 
Then we need to consider the following two matrices (when $|\mathcal{I}_1|=1$ and $|\mathcal{I}_1|=2$):
\begin{center}
\begin{tabular}{c@{\extracolsep{5mm}}c@{\extracolsep{5mm}}c}
$
L_1=
\left[
	\begin{array}{ccccccccc}
		A & -1 & -1\\
		-1 & 0 & -1\\
		-1 & B & -1\\
	\end{array}
\right] 
$
&
and
&
$
L_2=
\left[
	\begin{array}{ccccccccc}
		A & -1 & -1\\
		0 & -1 & -1\\
		-1 & B & -1\\
	\end{array}
\right], 
$
\end{tabular}
\end{center}
where $A$ is equal to $0$ (when $m\geq 2$ and $m\geq 3$, respectively) or $x_{i'_1}$ and 
$B$ is equal to $0$ (when $n\geq 3$ and $n\geq 2$, respectively) or $y_{i'_3}$.
It is not difficult to see that $\det(L_1)=AB-B$ and $\det(L_2)=AB-A$.
Thus, we get the minors $x_{i_1}y_{i'_3}-y_{i'_3}$ (when $n\geq 2$), $x_{i_1}y_{i'_3}-x_{i_1}$ (when $m\geq 2$), 
$-y_{i'_3}$ and $-x_{i_1}$ (when $m\geq 2$ and $n\geq 2$).
We get similar $3$-minors when $\mathcal{I}_2=\emptyset$ or $\mathcal{I}_1=\emptyset$.

Finally, consider the case when one of the $\mathcal{I}$'s is empty and one of the $\mathcal{J}'s$ is empty.
Assume that $\mathcal{I}_3=\emptyset$ and $\mathcal{J}_2=\emptyset$.
Then $|\mathcal{I}_2|=1$ and $L$ is equal to:
\[
\left[
	\begin{array}{ccccccccc}
		A & 0 & -1\\
		0 & A' & -1\\
		-1 & -1 & -1\\
	\end{array}
\right],
\]
where $A$ is equal to $0$ or $x_{i'_1}$ and $A'$ is equal to $0$ or $x_{i_2}$.
Clearly $\det L=-AA'-A-A'$.
Thus we get the $3$-minors $x_{i_1}x_{i_2}+x_{i_1}+x_{i_2}$ (when $m\geq 2$) and $x_{i_1}$ and $x_{i_2}$ (when $m\geq 3$).
Similarly when $\mathcal{J}_1=\emptyset$ and the other cases.
\end{proof}
The rest it follows by similar arguments to those in the case of the bipartite complete graph.
\end{proof}

\begin{proof}[Proof of theorem~\ref{teo:main1.3}]
Following the proof of theorem~\ref{teo:main1.2} we need to find the $3$-minors of the generalized Laplacian matrix of $T_n\vee (K_m+K_o)$.
We begin with $K_m\vee T_n$ and after that we do the same for $T_n\vee (K_m+K_o)$.
We omit the proofs of lemma~\ref{lemma:minorsJnm} and theorem~\ref{teo:minorsJmno}
because are rutinary and both follows by using similar arguments to those in 
lemma~\ref{lemma:minorsKmn} and in theorem~\ref{teo:minorsKmno}, respectively.

\begin{Lemma}\label{lemma:minorsJnm}
For $m,n\geq 1$, let $L'_{m,n}$ be the generalized Laplacian matrix of $K_m\vee T_n$. 
That is, 
\[
L'_{m,n}=L(K_m\vee T_n,\{ X_{K_n}, Y_{T_m}\})=\left[
\begin{array}{ccccccccc}
	L(K_{m},X_{K_{m}}) & -J_{m, n}\\
	-J_{n,m} & L(T_{n},Y_{T_{n}})\\
\end{array}
\right].
\]
Then the $3$-minors (with positive leading coefficient) of $L'_{m,n}$ are the following:

{\centering
\small
\begin{tabular}{l@{\extracolsep{2mm}}l} 
$\bullet$ {\scriptsize $y_{j_1}$, $y_{j_1} y_{j_2}$, and $y_{j_1} y_{j_2} y_{j_3}$ when $n\geq 3$}, &
$\bullet$ {\scriptsize $(x_{i_1}\!+\!1)(x_{i_2}\!+\!1)$, $(x_{i_1}\!+\!1)(y_{i_1}\!+\!1)$, and $x_{i_1} x_{i_2} x_{i_3}\!-\!x_{i_1}\!-\!x_{i_2}\!-\! x_{i_3}\!-\!2$ when $m\geq 3$},\\
$\bullet$ $x_{i_1} y_{j_1} y_{j_2}\!-\!y_{j_1}\!-\!y_{j_2}$  when $n\geq 2$, &
$\bullet$ $x_{i_1} x_{i_2} y_{j_1}\!-\!x_{i_1}\! -\! x_{i_2}\! -\! y_{j_1}\!-\!2$  when $m\geq 2$, \\
$\bullet$ $y_{j_1}$ when $m\geq 2$ and $n\geq 3$, &
$\bullet$ $x_{i_1}+1$ when $m\geq 3$ and $n\geq 2$, \\
\multicolumn{2}{l}{$\bullet$ $x_{i_1} + x_{i_2}+2$, $x_{i_1} + y_{j_1}$, $x_{i_1} y_{j_1}y_{j_2}$ and $y_{j_1}y_{j_2}+y_{j_1}+y_{j_2}$ when $m\geq 2$ and $n\geq 2$,}\\
\end{tabular}}

\noindent where $1\leq i_1< i_2< i_3\leq  m$ and $1\leq j_1<j_2<j_3\leq n$.
\end{Lemma}

\begin{Theorem}\label{teo:minorsJmno}
For $m,n,o\geq 1$, let $L'_{m,n,o}$ be the the generalized Laplacian matrix of $T_n\vee (K_m+K_o)$.
That is,
\[\small
L'_{m,n,o}=L(T_n\vee (K_m+K_o),\{X_{K_{m}},Y_{T_{n}},Z_{K_{o}}\})=
\left[
\begin{array}{ccccccccc}
	L(K_{m},X_{K_{m}}) & -J_{m, n} & {\bf 0}_{m, o}\\
	-J_{n, m} & L(T_{n},Y_{T_{n}}) & -J_{n, o}\\
	{\bf 0}_{o, m} & -J_{o, n} & L(K_{o},Z_{K_{o}})
\end{array}
\right].
\]
Then the $3$-minors (with positive leading coefficient) of $L'_{m,n,o}$ are the following:

{\centering
\small
\begin{tabular}{l@{\extracolsep{30mm}}l}
$\bullet$ $x_{i_1}+1$ when $m\geq 3$ and ($o\geq 2$ or $n\geq 2$), &
$\bullet$ $z_{k_1}+1$ when $o\geq 3$ and ($m\geq 2$ or $n\geq 2$),\\
$\bullet$ $y_{j_1}$ when $n\geq 3$ and ($m\geq 2$ or $o\geq 2$),&
$\bullet$ $2$ when $m\geq 2$, $n\geq 2$, and $o\geq 2$,\\
$\bullet$ $y_{j_1}$, $y_{j_1} y_{j_3}$, and $y_{j_1} y_{j_2} y_{j_3}$  when $n\geq 3$,&
$\bullet$ $x_{i_1} y_{j_1} z_{k_1}-x_{i_1}-z_{k_1}$,\\

\multicolumn{2}{l}{$\bullet$ $x_{i_1}+1$, $z_{k_1}(x_{i_1}+1)$, $(x_{i_1}+1)(x_{i_2}+1)$, $(x_{i_1}+1)(y_{j_1}+1)$,
and $x_{i_1} x_{i_2}x_{i_3}-x_{i_1}-x_{i_2}-x_{i_3}-2$, when $m\geq 3$,}\\
\multicolumn{2}{l}{$\bullet$ $z_{k_1}+1$, $x_{i_1}(z_{k_1}+1)$,  $(z_{k_1}+1)(z_{k_2}+1)$, $(z_{k_1}+1)(y_{j_1}+1)$,
and $z_{k_1} z_{k_2}z_{k_3}-z_{k_1}-z_{k_2}-z_{k_3}-2$ when $o\geq 3$,}\\

\multicolumn{2}{l}{$\bullet$ $x_{i_1}+z_{k_1}$, $y_{j_1}+y_{j_2}$, $x_{i_1} y_{j_1}$, $y_{j_1}z_{k_1}$, $x_{i_1} y_{j_1} y_{j_2}-y_{j_1}-y_{j_2}$, 
and $y_{j_1} y_{j_2}z_{k_1} -y_{j_1}-y_{j_2}$ when $n\geq 2$,}\\
\multicolumn{2}{l}{$\bullet$ {\footnotesize $x_{i_1}+1$, $x_{i_1} x_{i_2}-1$, $y_{j_1} z_{k_1}+z_{k_1}-1$, $z_{k_1}(x_{i_1}+1)$, $x_{i_1} x_{i_2}z_{k_1}-z_{k_1}$,
and $x_{i_1} x_{i_2} y_{j_1}-x_{i_1}-x_{i_2}-y_{j_1}-2$}, when $m\geq 2$,}\\
\multicolumn{2}{l}{$\bullet$ {\footnotesize $z_{k_1}+1$, $x_{i_1}(z_{k_1} z_{k_2}-1)$, $z_{k_1} z_{k_2}-1$, $x_{i_1}y_{j_1}+x_{i_1}-1$, $x_{i_1}(z_{k_1}+1)$,
and $z_{k_1} z_{k_2} y_{j_1}-z_{k_1}-z_{k_2}-y_{j_1}-2$}, when $o\geq 2$,}\\

\multicolumn{2}{l}{$\bullet$ $x_{i_1}+1$, $y_{j_1}+2$, $z_{k_1}+1$, $x_{i_1} x_{i_2}-1$, and $z_{k_1} z_{k_2}-1$ when $m\geq 2$ and $o\geq 2$,}\\
\multicolumn{2}{l}{$\bullet$ $x_{i_1}+1$, $y_{j_1}$, $z_{k_1}-1$, $x_{i_1}+y_{j_1}$, $x_{i_1}+x_{i_2}+2$, $y_{j_1}(x_{i_1}+1)$,
and $y_{j_1} y_{j_2}+y_{j_1}+y_{j_2}$  when $m\geq 2$ and $n\geq 2$,}\\
\multicolumn{2}{l}{$\bullet$ $x_{i_1}-1$, $y_{j_1}$, $z_{k_1}+1$, $z_{k_1}+y_{j_1}$, $z_{k_1}+z_{k_2}+2$, $y_{j_1}(z_{k_1} +1)$, and
$y_{j_1} y_{j_2}+y_{j_1}+y_{j_2}$, when $n\geq 2$ and $o\geq 2$,}\\
\end{tabular}}

\noindent where $1\leq i_1< i_2< i_3\leq  m$, $1\leq j_1<j_2<j_3\leq n$, and $1\leq k_1< k_2 <k_3\leq o$.
\end{Theorem}
\vspace{-7mm}
\end{proof}
%

Theorems~\ref{teo:main1.2} and~\ref{teo:main1.3} implies that ${\bf Forb}(\Gamma_{\leq 2})=\mathcal F_2$.
Now, we present the non-conected version of theorem~\ref{teo:main1}.

\begin{Corollary}\label{cor:main2}
A simple graph has algebraic co-rank equal to two if and only if is the disjoint union of a trivial graph with one of the following graphs:
\begin{itemize}
    	\item $K_{m,n,o}$, where $m\geq 2$, $n+o\geq 1$,
    	\item $T_n\vee (K_m+K_o)$, where $m, o\geq 2$, $m,n, o\geq 1$, or $n\geq 2$ and $m+o\geq 1$. 
\end{itemize}
\end{Corollary}
\begin{proof}
It is not difficult to see that in the non-connected case we need to add the 
graphs $P_3+P_2$ and $3P_2$ to the set of forbidden graphs.
The rest follows directly from theorem~\ref{teo:main1}.
\end{proof}

We finish this section with the classification of the graphs having critical group with 2 invariant factors equal to one.

\begin{Theorem}\label{teo:main2}
The critical group of a connected simple graph has exactly two invariant factor equal to $1$ if and only if is one of the following graphs:
\begin{itemize}
\item $K_{m,n,o}$, where $m\geq n\geq o$ satisfy one of the following conditions:
\begin{itemize}
\renewcommand{\labelitemii}{$\ast$}
	\item $m, n, o\geq 2$ with the same parity,
	\item $m, n\geq 3$, $o=1$, and $\gcd(m+1,n+1)\neq 1$,
	\item $m\geq 2$, $n=o=1$, 
	\item $m, n\geq 2$, $o=0$ and $\gcd(m,n)\neq 1$,
	\item $m\geq 2$, $n=2$, and $o=0$, or
	\item $m=2$ and $n=1$.
\end{itemize}
\item $T_n\vee (K_m+K_o)$, where $m\geq o$ and $n$ satisfy one of the following conditions:
\begin{itemize}
\renewcommand{\labelitemii}{$\ast$}
	\item $m, n, o\geq 2$ with the same parity,
	\item $m, o\geq 2$, $n=1$, and $\gcd(m+1,o+1) \neq 1$,
	\item $m, n\geq 2$, $o=1$, and $\gcd(m+1,n-1)\neq 1$,
	\item $m\geq 1$, $n=o=1$, 
	\item $n\geq 1$, $m=o=1$, 
	\item $m, n\geq 3$, $o=0$, and $\gcd(m,n)\neq 1$,
	\item $m\geq 2$, $n= 2$, $o=0$, or
	\item $m=2$, $n\geq 2$, $o=0$.
\end{itemize}
\end{itemize}
\end{Theorem}
\begin{proof}
It turns out from theorems~\ref{teo:main1.2} and~\ref{teo:main1.3}.
\end{proof}


\section{The set $\mathbf{Forb}(\Gamma_{\leq k})$.}
\label{sec:For}
The characterization of the $\gamma$-critical graphs with a given algebraic co-rank, ${\bf Forb}(\Gamma_{\leq k})$, is very important.
For instance, we were able to characterize $\Gamma_{\leq k}$ for $k$ equal to $1$ and $2$ because
we got a finite set of $\gamma$-critical graphs with algebraic co-rank equal to $k+1$ (for $k$ equal to $1$ and $2$), and
after that we proved that all the graphs that do not contain a graph from this set as an induced subgraph
has algebraic co-rank less or equal to $k$.
In this section we give two infinite families of forbidden simple graphs.
This will prove that ${\bf Forb}(\Gamma_{\leq k})$ is not empty for all $k\geq 0$.
Moreover, we conjecture that ${\bf Forb}(\Gamma_{\leq k})$ is finite for all $k\geq 0$.
To finish we present an example of a simple graph $G$ with algebraic co-rank equal to $5$ but with no $5$-minor equal to $1$. 
That is, the $1$ can be obtained uniquely from a non trivial algebraic combination of $5$-minors of $L(G,X)$.

\medskip

We begin by proving that the path with $n+2$ vertices is $\gamma$-critical with algebraic co-rank equal to $n+1$.
\begin{Theorem}\label{camino}
If $n\geq 0$, then $P_{n+2}\in {\bf Forb}(\Gamma_{\leq n})$. 
\end{Theorem}
\begin{proof}
It is not difficult to prove $\gamma(P_{n+2})=n+1$, see corollary 4.10 of \cite{corrval}.
On the other hand, if $H=P_{n+2}\setminus v$ for some $v\in V(P_{n+2})$, then $H$ is a disjoint union of at most two paths.
Let $H=P_{n_1}+\cdots+P_{n_s}$ with $1\leq s \leq 2$ and $\sum_{i=}^s n_i=n+1$, then by
lemma \ref{lemm:corrvaldisjoint} we get that 
\[
\gamma(H)=\sum_{i=1}^s \gamma(P_{n_i})=\sum_{i=1}^s (n_i-1)=\sum_{i=1}^s n_i-s=n+1-s< n+1.
\]
Therefore $P_{n+2}\in {\bf Forb}(\Gamma_{\leq n})$.
\end{proof}

Now, we present another infinite family of graph that are $\gamma$-critical.
Let $K_n$ be the complete graph with $n$ vertices and $M_k$ a matching of $K_n$ with $k$ edges.
We begin by finding the critical group of $K_n\setminus M_k$.  

\begin{Proposition}\label{prop:match}
If $K_n$ be the complete graph with $n$ vertices and $M_k$ is a matching of $k$ edges, then
\[
K(K_n\setminus M_k)
\cong
\begin{cases}
	\mathbb{Z}_n^{n-2k-2}\oplus \mathbb{Z}_{n(n-2)}^k  & \text{if }n\geq 2k+2,\\
	 \mathbb{Z}_{n-2} \oplus \mathbb{Z}_{n(n-2)}^{k-1} & \text{if }n= 2k+1.\\
\end{cases}
\]
\end{Proposition}
\begin{proof}
If $n= 2k+1$ the result follows by~\cite[Theorem 1]{reiner2003}.
Therefore we can assume that $n\geq 2k+2$.
Given ${\bf a}\in \mathbb{Z}^k$, let $N_{k+1}({\bf a})$ be the matrix given by
\[
\left[
\begin{array}{cc}
1 & { \bf a }\\
{ \bf 0 }^t  & I_{k}\\
\end{array}
\right].
\]
If $M_k=\{v_1v_2,\ldots, v_{2k-1}v_{2k}\}$, then
\[
L(K_n\setminus M_k,v_n)
=
\left[
\begin{array}{cc}
	  [(n-2)I_2+J_2]\otimes I_k-J_{2k}    &     -J_{2k,n-2k-1}    \\
	    -J_{n-2k-1,2k}    &   nI_{n-2k-1}-J_{n-2k-1}     \\
\end{array}
\right],
\]
where $\otimes$ is the tensor product of matrices.
Now, since ${\rm det}(N_{n-1}({\bf a}))=1$ for all ${\bf a}$, then
\begin{eqnarray*}
L(K_n\setminus M_k,v_n) &\sim&N_{n-1}({\bf 1})^t N_{n-1}({\bf 1}) L(K_n\setminus M_k,v_n) N_{n-1}({\bf -1})\\
&=&I_{1} \oplus n I_{n-2k-2} \bigoplus_{i=1}^k 
\left[
\begin{array}{cc}
	  n-1 & 1 \\
	  1 & n-1 \\
\end{array}
\right].	
\end{eqnarray*}
On the other hand 
\begin{eqnarray*}
\left[
\begin{array}{cc}
	  n-1 & 1 \\
	  1 & n-1 \\
\end{array}
\right] & \sim & 
\left[
\begin{array}{cc}
	  0 & 1 \\
	  -1 & n-1 \\
\end{array}
\right]
\left[
\begin{array}{cc}
	  n-1 & 1 \\
	  1 & n-1 \\
\end{array}
\right]
\left[
\begin{array}{cc}
	  1 & -(n-1) \\
	  0 & 1 \\
\end{array}
\right]\\
& =&
\left[
\begin{array}{cc}
	  1 & 0 \\
	  0 & n(n-2) \\
\end{array}
\right].
\end{eqnarray*}

\noindent Therefore	$L(K_n\setminus M_k,v_n)\sim I_{k+1} \oplus n I_{n-2k-2} \oplus n(n-2)I_k$.
\end{proof}

\begin{Corollary}\label{cor:match}
If $n=2k+2$, then $K_n\setminus M_k\in {\bf Forb}(\Gamma_{\leq k})$.
\end{Corollary}
\begin{proof}
First, by proposition~\ref{prop:match} we have that 
\[
\gamma(K_n\setminus M_k)\leq 
\begin{cases}
k+1 & \text{ if } n\geq 2k+2,\\
k & \text{ if } n= 2k+1.
\end{cases}
\] 
Now, let $n\geq 2k+2$, $M_k=\{v_1v_2,\ldots, v_{2k-1}v_{2k}\}$, and $M=L(K_n\setminus M_k,X)[\{1,\ldots,2k+1\},\{2,\ldots,2k+2\}]$
be a square submatrix of generalized Laplacian matrix of $K_n\setminus M_k$.
Then
\[
M=
\left[
\begin{array}{ccccc}
	  0 & -1  &   & -1 & -1\\ 
	 -1 & 0  &  &  -1 & -1\\ 
	  &  &\ddots &  -1 & -1 \\ 
	 -1 & -1 & -1 & 0  & -1\\
	 -1 &  -1 & -1 & -1 & -1 \\
\end{array}
\right].
\]
By~\cite[theorem 3.13]{corrval}, ${\rm det}(M)={\rm det}(L(K_k,X_{K_k}))|_{\{x_1=0, \ldots, x_{k-1}=0, x_k=-1\}}\overset{3.13}{=}-1$
and therefore, $\gamma(K_n\setminus M_k)= k+1$ for all $n\geq 2k+2$.
Finally, if $n= 2k+2$ and $v\in V(K_n\setminus M_k)$, then $(K_n\setminus M_k)\setminus v$ 
is equal to $K_{n-1}\setminus M_k$ or $K_{n-1}\setminus M_{k-1}$.
Thus, $\gamma((K_n\setminus M_k)\setminus v)\leq k$ and therefore $K_n\setminus M_k\in {\bf Forb}(\Gamma_{\leq k})$.
\end{proof}

This results proves that ${\bf Forb}(\Gamma_{\leq k})$ is not empty for all $k\geq 0$.

\begin{Corollary}
If $k\geq 0$, then ${\bf Forb}(\Gamma_{\leq k})$ is not empty.
\end{Corollary}

For $i\geq 3$, the set ${\bf Forb}(\Gamma_{\leq i})$ is more complex than ${\bf Forb}(\Gamma_{\leq 1})$ and ${\bf Forb}(\Gamma_{\leq 2})$.
For instance, in~\cite{alfval3} was proved that ${\bf Forb}(\Gamma_{\leq 3})$ has 
$49$ graphs.
Moreover, we conjecture that ${\bf Forb}(\Gamma_{\leq k})$ is finite. 

\begin{Conjecture}\label{conj:finite}
For all $k\in \mathbb{N}$ the set ${\bf Forb}(\Gamma_{\leq k})$ is finite.
\end{Conjecture}

Until now, all the graphs that were presented has algebraic co-rank equal to $k$ because its generalized Laplacian matrix has a $k$-minor equal to one.
Next example shows a graph $G$ with $\gamma_\mathbb{Z}(G)=5$ having no a $5$-minor equal to $1$.

\begin{Example}
Let $G$ be the graph on figure~\ref{fige} and
\begin{figure}[h]
\begin{center}
\begin{tabular}{c@{\extracolsep{2cm}}c}
\multirow{9}{3cm}{
	\begin{tikzpicture}[line width=1pt, scale=1]
		\tikzstyle{every node}=[inner sep=0pt, minimum width=4.5pt] 
		\draw (135:1) node (v1) [draw, circle, fill=gray] {};
		\draw (225:1)+(-1,0) node (v2) [draw, circle, fill=gray] {};
		\draw (225:1) node (v3) [draw, circle, fill=gray] {};
		\draw (315:1) node (v4) [draw, circle, fill=gray] {};
		\draw (315:1)+(1,0) node (v5) [draw, circle, fill=gray] {};
		\draw (45:1) node (v6) [draw, circle, fill=gray] {};
		\draw (0,-2) node (v7) [draw, circle, fill=gray] {};
		\draw (v1)+(-0.2,0.2) node () {\small $v_1$};
		\draw (v2)+(0,-0.3) node () {\small $v_2$};
		\draw (v3)+(-0.1,-0.3) node () {\small $v_3$};
		\draw (v4)+(0.1,-0.3) node () {\small $v_4$};
		\draw (v5)+(0,-0.3) node () {\small $v_5$};
		\draw (v6)+(0.2,0.2) node () {\small $v_6$};
		\draw (0,-0.5) node () {\small $G$};
		\draw (v1) -- (v3) -- (v4) -- (v6) -- (v1);
		\draw (v1) -- (v2) -- (v3);
		\draw (v6) -- (v5) -- (v4);
		\draw (v3) edge (v6);
		\draw (v7) edge (v1);
		\draw (v7) edge (v2);
		\draw (v7) edge (v3);
		\draw (v7) edge (v4);
		\draw (v7) edge (v5);
		\draw (v7) edge (v6);
	\end{tikzpicture}
}
& \\
&
$
L(G, X)=
\left[\begin{array}{ccccccc}
x_1 & \cellcolor[gray]{0.7}-1 &  \cellcolor[gray]{0.7}-1 &  0 & \cellcolor[gray]{0.7}0 & \cellcolor[gray]{0.7}-1 & \cellcolor[gray]{0.7}-1\\
 -1 & \cellcolor[gray]{0.7}x_2 & \cellcolor[gray]{0.7}-1 & 0 & \cellcolor[gray]{0.7}0 & \cellcolor[gray]{0.7}0 & \cellcolor[gray]{0.7}-1\\
 -1 & \cellcolor[gray]{0.7}-1 &  \cellcolor[gray]{0.7}x_3 & -1 & \cellcolor[gray]{0.7}0 & \cellcolor[gray]{0.7}-1 & \cellcolor[gray]{0.7}-1\\
0 &   \cellcolor[gray]{0.7}0 &  \cellcolor[gray]{0.7}-1 &  x_4 & \cellcolor[gray]{0.7}-1 & \cellcolor[gray]{0.7}-1 & \cellcolor[gray]{0.7}-1\\
0  & \cellcolor[gray]{0.7}0 &  \cellcolor[gray]{0.7}0 &  -1 & \cellcolor[gray]{0.7}x_5 & \cellcolor[gray]{0.7}-1 & \cellcolor[gray]{0.7}-1\\
-1 & 0 &  -1 &  -1 & -1 & x_6 & -1\\
-1 & -1 &  -1 &  -1 & -1 & -1 & x_7
\end{array}\right]
$\\
& \\
\end{tabular}
\end{center}
\caption{A graph $G$ with seven vertices and its generalized Laplacian matrix.}
\label{fige}
\end{figure}
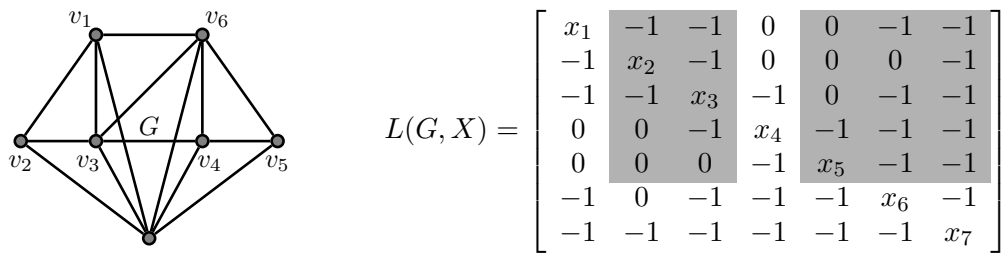
$f_1={\rm det}(L(G,X)[\{1, 2, 3,4, 5\}; \{2, 3, 5, 6, 7\}]) =x_2+x_5+x_2x_5$, and 
$f_2={\rm det}(L(G,X)[\{1, 2, 3, 5, 6\}; \{ 2,4, 5, 6, 7\}])=-(1+x_2+x_5+x_2x_5)$. 
Then $\langle f_1, f_2\rangle=1$ and therefore $\gamma_\mathbb{Z}(G)=5$. 
However, it is not difficult to check that $L(G,X)$ has no $5$-minor is equal to one.
\end{Example}


\end{document}